\documentclass[11pt]{article}

\usepackage{enumerate}
\usepackage{amsmath}
\usepackage{amssymb,latexsym}
\usepackage{amsthm}
\usepackage{color}
\usepackage{graphicx}
\usepackage{hyperref}
\usepackage[notcite, notref]{}
\usepackage{amscd}

\DeclareMathOperator{\N}{N}

\title{Maximum scattered linear sets and MRD-codes}
\author{Bence Csajb\'ok, Giuseppe Marino, Olga Polverino, Ferdinando Zullo \thanks{\textcolor{black}{The
research  was supported by
Ministry for Education, University and Research of Italy MIUR (Project
PRIN 2012 "Geometrie di Galois e strutture di incidenza") and by the Italian National
Group for Algebraic and Geometric Structures and their Applications (GNSAGA
- INdAM).}}}
\date{ }

\newcommand{\cC}{{\mathcal C}}

\newcommand{\cG}{{\mathcal G}}

\newcommand{\cD}{{\mathcal D}}

\newcommand{\cS}{{\mathcal S}}

\newcommand{\F}{{\mathbb F}}

\newcommand{\la}{\langle}
\newcommand{\ra}{\rangle}

\newcommand{\bv}{\mathbf v}

\newtheorem{theorem}{Theorem}[section]

\newtheorem{lemma}[theorem]{Lemma}

\newtheorem{definition}[theorem]{Definition}
\newtheorem{proposition}[theorem]{Proposition}

\newtheorem{example}[theorem]{Example}

\newtheorem{remark}[theorem]{Remark}

\DeclareMathOperator{\PG}{{PG}}

\DeclareMathOperator{\GL}{{GL}}

\begin{document}
\maketitle

\begin{abstract}
The rank of a scattered $\F_q$-linear set of $\PG(r-1,q^n)$, $rn$ even, is at most $rn/2$ as it was proved by Blokhuis and Lavrauw.
Existence results and explicit constructions were given for infinitely many values of $r$, $n$, $q$ ($rn$ even) for scattered $\F_q$-linear
sets of rank $rn/2$. In this paper we prove that the bound $rn/2$ is sharp also in the remaining open cases.

Recently Sheekey proved that scattered $\F_q$-linear sets of $\PG(1,q^n)$ of maximum rank $n$ yield $\F_q$--linear MRD-codes with dimension $2n$ and minimum distance $n-1$.
We generalize this result and show that scattered $\F_q$-linear sets of $\PG(r-1,q^n)$ of maximum rank $rn/2$ yield $\F_q$--linear MRD-codes with dimension $rn$ and minimum distance $n-1$.
\end{abstract}



\section{Introduction}
Let $\Lambda=\PG(V,\F_{q^n})=\PG(r-1,q^n)$, $q=p^h$, $p$ prime, $V$ a vector space of dimension $r$ over $\F_{q^n}$, and let
$L$ be a set of points of $\Lambda$. The set $L$ is said to be an
{\it $\F_q$--linear} set of $\Lambda$ of rank $k$ if it is defined by
the non-zero vectors of an $\F_q$-vector subspace $U$ of $V$ of
dimension $k$, i.e.
\begin{equation}\label{inizio}L=L_U=\{\langle {\bf u}\rangle_{\F_{q^n}}: {\bf u}\in
U\setminus\{{\bf 0}\}\}.
\end{equation}
We point out that different vector subspaces can define the same linear set. For this reason a linear set and the vector space defining it must be considered as coming in pair.

Let  $\Omega=\PG(W,\F_{q^n})$ be a subspace of $\Lambda$ and let $L_U$
be an $\F_q$-linear set of $\Lambda$. Then $\Omega\cap L_U$ is an
$\F_q$--linear set of $\Omega$ defined by the $\F_q$--vector
subspace $U\cap W$ and, if $\dim_{\F_q}(W\cap U)=i$, we say that
$\Omega$ has {\it weight $i$} in $L_U$. Hence a point of $\Lambda$ belongs to $L_U$ if and only if it has weight at least 1 and if $L_U$ has rank $k$, then $|L_U|\leq  q^{k-1}+q^{k-2}+\dots+q+1$. For further details on linear sets see \cite{OP2010}, \cite{LaVa2010}, \cite{LV}, \cite{LuMaPoTr2014}, \cite{LuPo2004}, \cite{LV2013}, \cite{CSZ2015} and \cite{CSZ2016}.

An $\F_q$--linear set $L_U$ of $\Lambda$ of rank $k$ is {\em scattered} if all of its points have weight 1, or equivalently, if
$L_U$ has maximum size $q^{k-1}+q^{k-2}+\cdots+q+1$.
A scattered $\F_q$--linear set of $\Lambda$ of highest possible
rank is a {\it maximum scattered $\F_q$--linear set} of $\Lambda$; see \cite{BL2000}. Maximum scattered linear sets have a lot of applications in Galois Geometry, such as translation hyperovals \cite{Glynn}, translation caps in affine spaces \cite{BGMP2015}, two-intersection sets (\cite{BL2000}, \cite{BL}), blocking sets (\cite{PP},
\cite{LP2000}, \cite{L2001} \cite{BP2005},
\cite{BBL2000}), translation spreads of the Cayley
generalized hexagon (\cite{CLPT},
\cite{BP2005TC}, \cite{MaPo2015}), finite semifields (see e.g. \cite{L2003}, \cite{CPT},
\cite{MPT2007}, \cite{EMPT1}, \cite{LuMaPoTr2014}, \cite{LMPT}, \cite{LaMaPoTr2013}, \cite{LaMaPoTr2013-1}), coding theory and graph theory \cite{CK}. For a recent survey on the theory of scattered spaces in Galois Geometry and its applications see \cite{Lavrauw}.

The rank of a scattered $\F_q$-linear set of $\PG(r-1,q^n)$, $rn$ even, is at most $rn/2$ (\cite[Theorems 2.1, 4.2 and 4.3]{BL2000}).
For $n=2$ scattered $\F_q$-linear sets of $\PG(r-1,q^2)$ of rank $r$ are the Baer subgeometries.
When $r$ is even there always exist scattered $\F_q$--linear sets of rank
$\frac{rn}2$ in $\PG(r-1,q^n)$, for any $n\geq 2$ (see \cite[Theorem 2.5.5]{LPhdThesis} for an explicit example).
Existence results were proved for $r$ odd, $n-1 \leq r$, $n$ even, and $q>2$ in \cite[Theorem 4.4]{BL2000}, but
no explicit constructions were known for $r$ odd, except for the case $r=3$, $n=4$, see \cite[Section 3]{BBL2000}.
Very recently families of scattered linear sets of rank $rn/2$ in $\PG(r-1,q^n)$, $r$ odd, $n$ even, were constructed in \cite[Theorem 1.2]{BGMP2015} for infinitely many values of $r$, $n$ and $q$.

The existence of scattered $\F_q$--linear sets of rank $\frac{3n}2$ in $\PG(2,q^n)$, $n\geq 6$ even, $n\equiv 0 \pmod 3$, $q\not\equiv 1\pmod 3$ and $q>2$ was posed as an open problem in \cite[Section 4]{BGMP2015}.
As it was pointed out in \cite{BGMP2015}, the existence of such planar linear sets and the construction method of \cite[Theorem 3.1]{BGMP2015} would imply that the bound $\frac{rn}2$ for the maximum rank of a scattered $\F_q$--linear set in $\PG(r-1,q^n)$ is
also tight when $r$ is odd and $n$ is even.
In Theorem \ref{thm:1} we construct linear sets of rank $3n/2$ of $\PG(2,q^n)$, $n$ even, and hence we prove the sharpness of the bound also in the remaining open cases. Our construction relies on the existence of non-scattered linear sets of rank $3t$ of $\PG(1,q^{3t})$ (with $t=n/2$) defined by binomial polynomials. 

\medskip

In \cite[Section 4]{Sh} Sheekey showed that maximum scattered $\F_q$-linear sets of $\PG(1,q^n)$ correspond to $\F_q$-linear maximum rank distance codes (MRD-codes) of dimension $2n$ and minimum distance $n-1$.
In Section 3 we extend this result showing that MRD-codes can be constructed from every scattered linear set of rank $rn/2$ of $\PG(r-1,q^n)$, $rn$ even, and
we point out some relations with Sheekey's construction. Finally, we exhibit the MRD-codes arising from maximum scattered linear sets constructed in Theorem \ref{thm:1} and those constructed in \cite[Theorems 2.2 and 2.3]{BGMP2015}

\section{\texorpdfstring{Maximum scattered linear sets in $\PG(r-1,q^n)$}{Maximum scattered linear sets in PG(r-1,qn)}}

As it was pointed out in the Introduction, the existence of scattered $\F_q$--linear sets of rank $\frac{3n}2$ in
$\PG(2,q^n)$, $n\geq 6$ even, $n\equiv 0 \pmod 3$, $q\not\equiv 1\pmod 3$ and $q>2$ would imply that the bound $\frac{rn}2$ for the rank of a maximum scattered $\F_q$--linear set in $\PG(r-1,q^n)$ is tight in the remaining open cases (cf. \cite[Remark 2.11 and Section 4]{BGMP2015}).

\medskip

In this section we show that binomials of the form $f(x)=ax^{q^i}+bx^{2t+i}$ defined over $\F_{q^{3t}}$ can be used to construct maximum scattered $\F_q$--linear sets in $\PG(2,q^{2t})$ for any $t\geq 2$ and for any prime power $q$.

\bigskip

Consider the finite field $\F_{q^{6t}}$ as a $3$--dimensional vector space over its subfield $\F_{q^{2t}}$, $t\geq 2$, and let
${\mathbb P}=\PG(\F_{q^{6t}},\F_{q^{2t}})=\PG(2,q^{2t})$ be the associated projective plane. From \cite[Section 2.2]{BGMP2015}, the $\F_q$-subspace
\begin{equation}\label{form:subspace}
U:=\{\omega x+f(x) \colon x\in \F_{q^{3t}}\},
\end{equation}
of $\F_{q^{6t}}$ with $\omega\in\F_{q^{2t}}\setminus\F_{q^t}$, $f(x)=ax^{q^i}+bx^{q^{2t+i}}$, $a,b\in\F_{q^{3t}}^*$, $1\leq i\leq 3t-1$ and $gcd(i,2t)=1$, defines a maximum scattered $\F_q$-linear set in the projective plane ${\mathbb P}$ of rank $3t$ if $\frac{f(x)}{x}\notin\F_{q^t}$ for each $x\in\F_{q^{3t}}^*$ (cf. \cite[Prop. 2.7]{BGMP2015}).
The $q$-polynomial $f(x)$ also defines an $\F_q$-linear set $L_f:=\{\la (x,f(x)) \ra_{\F_{q^{3t}}} \colon x\in \F_{q^{3t}}^*\}$ of the projective line $\PG(\F_{q^{6t}},\F_{q^{3t}})=\PG(1,q^{3t})$. In what follows we determine some conditions on $L_f$ in order to obtain maximum scattered $\F_q$-linear sets in ${\mathbb P}$ of rank $3t$.

If $h \mid n$, then by $\N_{q^n/q^h}(\alpha)$ we will denote the norm of $\alpha\in\F_{q^n}$ over the subfield $\F_{q^h}$, that is,
$\N_{q^n/q^h}(\alpha)=\alpha^{1+q^h+\ldots+q^{n-h}}$. We will need the following preliminary result.

\begin{lemma}
\label{prop1}
Let $f:=f_{i,a,b}:x\in\F_{q^{3t}}\mapsto ax^{q^i}+bx^{q^{2t+i}}\in\F_{q^{3t}}$, with $a,b\in \F_{q^{3t}}^*$, $\N_{q^{3t}/q^t}(a)\neq -\N_{q^{3t}/q^t}(b)$ and $\gcd(i,t)=1$.
If \begin{equation}\label{form:Lline} L_f:=\{\la (x,f(x)) \ra_{\F_{q^{3t}}} \colon x\in \F_{q^{3t}}^*\}\end{equation} is not a scattered $\F_q$--linear set of $\PG(1,q^{3t})$, then there exists $c\in \F_{q^{3t}}^*$ such that
\begin{equation}\label{eq3}
g_c(x):=\frac{f_{i,ca,cb}(x)}{x}\notin \F_{q^t}\quad \text{for each }x\in \mathbb F_{q^{3t}}^*.
\end{equation}
\end{lemma}
\begin{proof}
First we show that $0 \notin Im\, g_c$ for each $c$.
If $cax_0^{q^i-1}=-cbx_0^{q^{2t+i}-1}$ for some $x_0\in \F_{q^{3t}}^*$, then $-a/b=x_0^{q^i(q^{2t}-1)}$, where
the right hand side is a $(q^t-1)$-th power and hence $\N_{q^{3t}/q^t}(-a/b)=1$, a contradiction.

The non-zero elements of the one-dimensional $\F_{q^t}$-spaces of $\F_{q^{3t}}^*$ yield a partition of $\F_{q^{3t}}^*$ into $q^{2t}+q^t+1$
subsets of size $q^t-1$. More precisely, if $\mu$ is a primitive element of $\F_{q^{3t}}$, then
\[\F_{q^{3t}}^*=\bigcup_{k=0}^{q^{2t}+q^t} \mu^{k} \F_{q^t}^*.\]
Let $G_k:=\mu^k \F_{q^t}^*$.
We show that, for each $k$, either $Im\, g_1 \cap G_k = \emptyset$, or $|Im\, g_1 \cap G_k| \geq (q^t-1)/(q-1)$.

Suppose $g_1(x_0) \in G_k$. Then for each $\gamma \in \F_{q^t}^*$ we have
\[g_1(\gamma x_0)=\gamma^{q^i-1}g_1(x_0).\]
Since $\gcd(i,t)=1$, it follows that $$\{g_1(\gamma x_0) \colon \gamma\in \F_{q^t}^*\}=g_1(x_0) \{x \in \F_{q^t} \colon \N_{q^t/q}(x)=1\} \subseteq G_k$$ and hence $|Im\, g_1 \cap G_k| \geq (q^t-1)/(q-1)$.

Next we show that there exists $G_d$ such that $Im\, g_1 \cap G_d=\emptyset$.
Suppose to the contrary $Im\, g_1 \cap G_j \neq \emptyset$ for each $j\in \{0,1,\ldots,q^{2t}+q^t\}$.
Then $|Im\, g_1|\geq (q^{2t}+q^t+1)(q^t-1)/(q-1)=(q^{3t}-1)/(q-1)$ and since $|Im\, g_1|=|L_f|$ we get a contradiction.

Suppose that $Im\, g_1 \cap G_d = \emptyset$ and let $c=\mu^{-d}$. Then $Im\, g_c \cap \F_{q^t} = \emptyset$.
\end{proof}

Hence, by the previous lemma and by \cite[Prop. 2.7]{BGMP2015}, the existence of a non-scattered linear set in $\PG(1,q^{3t})$ of form (\ref{form:Lline}) implies the existence of a binomial polynomial producing maximum scattered $\F_q$-linear set in $\PG(2,q^{2t})$ of rank $3t$.

\begin{lemma}
\label{thm:nonscattered}
Let $f:=f_{i,a,b} \colon x\in\F_{q^{3t}}\mapsto ax^{q^i}+bx^{q^{2t+i}}\in\F_{q^{3t}}$, with $a,b\in \F_{q^{3t}}^*$ and $1\leq i\leq 3t-1$.
For any prime power $q\geq 2$ and any integer $t\geq 2$ there exist $a,b\in\F_{q^{3t}}^*$, with
\begin{equation}
\label{cond1}
\N_{q^{3t}/q^t}(b)\ne -\N_{q^{3t}/q^t}(a),
\end{equation}
such that
\[L_{f_{i,a,b}}:=\{\la (x,f_{i,a,b}(x)) \ra_{\F_{q^{3t}}} \colon x\in \F_{q^{3t}}^*\},\]
is a non--scattered $\F_q$--linear set in $\PG(1,q^{3t})$ of rank $3t$.
\end{lemma}

\begin{proof}
First suppose $d:=\gcd(i,t)>1$. Then $f$ is $\F_{q^d}$-linear and hence each point of $L_f$ has wight at least $d$, i.e. $L_f$ cannot be scattered. Since $q^t\geq 4$ we can always choose $a,b$ such that \eqref{cond1} holds. From now on we assume $\gcd(i,t)=1$.

The linear set $L_f$ of $\PG(1,q^{3t})$ is not scattered if there exists a point $P_{x_0}=\la(x_0,f(x_0))\ra_{\F_{q^{3t}}}$ of rank greater than 1, i.e. if there exist $x_0\in{\F_{q^{3t}}}^*$ and $\lambda\in{\F_{q^{3t}}}\setminus\F_q$ such that $f(\lambda x_0)=\lambda f(x_0)$.
The latter condition is equivalent to
\begin{equation}
\label{nonz}
ax_0^{q^i}(\lambda-\lambda^{q^i})=bx_0^{q^{2t+i}}(\lambda^{q^{2t+i}}-\lambda).
\end{equation}
Since $\gcd(2t+i,3t), \gcd(i,3t)\in \{1,3\}$, the expressions in the two sides of \eqref{nonz} are non-zero when $\lambda\notin \F_{q^3}$.
We first prove that there exists $\bar\lambda\in{\F_{q^{3t}}}\setminus\F_{q^3}$ such that
\begin{equation}\label{form:alpha-lambda}
\N_{{q^{3t}}/{q^t}}(\alpha_{\bar\lambda})\ne -1,
\end{equation}
where $\alpha_{\bar\lambda}=\frac{\bar\lambda-\bar\lambda^{q^i}}{\bar\lambda^{q^{2t+i}}-\bar\lambda}$\,.

By way of contradiction, suppose that $\N_{{q^{3t}}/{q^t}}(\alpha_{\bar\lambda})= -1$ for each $\bar\lambda\in{\F_{q^{3t}}}\setminus\F_{q^3}$. Then the polynomial
\begin{equation}\label{form:polynomial-g}
g(x):=(x-x^{q^i})(x^{q^t}-x^{q^{t+i}})(x^{q^{2t}}-x^{q^{i+2t}})+(x^{q^{2t+i}}-x)(x^{q^i}-x^{q^t})(x^{q^{t+i}}-x^{q^{2t}})
\end{equation}
vanishes on $\F_{q^{3t}}\setminus \F_{q^3}$. It also vanishes on $\F_q$, thus it has at least $q^{3t}-q^3+q$ roots. Put $i=c+mt$, with $m\in\{0,1,2\}$ and $1\leq c<t$, the degree of $g(x)$ is
\begin{equation}
\label{0form:degree}
q^{2t+c}+q^{2t}+q^t
\end{equation}
when $m=0$ and
\begin{equation}\label{form:degree}
q^{2t+c}+q^{2t}+q^{t+c}
\end{equation}
when $m\in\{1,2\}$.
Since $q^t-2 \geq q^c$ we obtain
\[q^{2t+c}+q^{2t}+q^{t+c}=q^{c}(q^{2t}+q^t)+q^{2t} \leq (q^t-2)(q^{2t}+q^t)+q^{2t} = q^{3t}-2q^t.\]
For $t>2$ this is a contradiction since $q^{3t}-2q^t < q^{3t}-q^3+q$.
If $t=2$, then $\gcd(i,t)=1$ yields $c=1$ and hence we obtain
\[\deg g \leq q^5+q^4+q^3 < q^6-q^3+q,\]
again a contradiction.
It follows that there always exists an element $\bar\lambda\in\F_{q^{3t}}\setminus\F_{q^3}$ which is not a root of $g(x)$, and $\alpha_{\bar \lambda}$ satisfies Condition (\ref{form:alpha-lambda}).

Choose $a,b\in\F_{q^{3t}}^*$ such that $\N_{{q^{3t}}/{q^t}}(\frac ba)=\N_{{q^{3t}}/{q^t}}(\alpha_{\bar\lambda})$, then there exists an element $x_0\in\F_{q^{3t}}^*$ such that
\[x_0^{q^{2t+i}-q^i}=\frac ab \alpha_{\bar\lambda},\]
and hence $x_0$ is a non-zero solution of the equation $f(\bar\lambda x)=\bar\lambda f(x)$, i.e. with these choices of $a$ and $b$ the linear set $L_{f_{i,a,b}}$ is not scattered.
\end{proof}


Now we are able to prove the following result.

\begin{theorem}
\label{thm:1}
Let $w \in \F_{q^{2t}}\setminus \F_{q^t}$. For any prime power $q$ and any integer $t\geq 2$, there exist $a, b \in \F_{q^{3t}}^*$ and an integer $1 \leq i \leq 3t-1$ such that the $\F_q$-linear set $L_U$ of rank $3t$ of the projective plane $\PG(\F_{q^{6t}},\F_{q^{2t}})=\PG(2,q^{2t})$, where
\[U=\{ ax^{q^i}+bx^{q^{2t+i}}+w x \colon x \in \F_{q^{3t}}\},\]
is scattered.
\end{theorem}
\begin{proof}
According to Lemma \ref{thm:nonscattered} for any prime power $q$ and any integers $t\geq 2$, $1 \leq i \leq 3t-1$ with $\gcd(i,2t)=1$ we can choose ${\bar a}, {\bar b} \in \F_{q^{3t}}^*$, with $\N_{q^{3t}/q^t}({\bar b}) \neq -\N_{q^{3t}/q^t}({\bar a})$ such that the linear set $L_f$ of the line $\PG(\F_{q^{6t}},\F_{q^{3t}})=\PG(1,q^3)$ with $f(x)={\bar a}x^{q^i}+{\bar b}x^{q^{2t+i}}$ is non-scattered.
Then by Lemma \ref{prop1} there exists $c\in \F_{q^{3t}}^*$ such that
\[\frac{{\bar a}cx^{q^i}+{\bar b}cx^{q^{2t+i}}}{x} \notin \F_{q^t}\]
for each $x\in \F_{q^{3t}}^*$. Then the theorem follows from \cite[Proposition 2.7]{BGMP2015} with $a={\bar a}c$ and $b={\bar b}c$.
\end{proof}

As it was pointed out in \cite{BGMP2015}, the existence of maximum scattered $\F_q$--linear sets of rank $3n$ in the projective plane $\PG(2,q^{2t})$
(proved in Theorem \ref{thm:1}) and the construction method of \cite[Theorem 3.1]{BGMP2015} imply the following.

\begin{theorem}\label{cor:existence-scatt}
For any integers $r, n\geq 2$, $rn$ even, and for any prime power $q\geq 2$ the rank of a maximum scattered $\F_q$-linear set of $\PG(r-1,q^n)$ is $rn/2$.
\end{theorem}

Taking into account the previous result, from now on, a scattered $\F_q$--linear set $L_U$ of $\PG(W,\F_{q^n})=\PG(r-1,q^n)$ of rank $\frac{rn}2$ ($rn$ even) will be simply called a {\em maximum scattered linear set} and the $\F_q$-subspace $U$ will be called a {\em maximum scattered subspace}.

\bigskip

We complete this section by showing a connection between scattered $\F_q$-linear sets of $\PG(1,q^{rn/2})$, $r$ even, and scattered $\F_q$-linear sets of $\PG(r-1,q^n)$.

\begin{proposition}
\label{prop2}
Every maximum scattered $\F_q$-linear set of $\PG(1,q^{rn/2})$, $r$ even, gives a maximum scattered $\F_q$-linear set of $\PG(r-1,q^n)$.
\end{proposition}
\begin{proof}
Let $L_U$ be a maximum scattered $\F_q$-linear set of $\PG(W,\F_{q^{rn/2}})=\PG(1,q^{rn/2})$.
Then for each ${\bf v} \in W$ the one dimensional $\F_{q^{rn/2}}$-subspace $\la {\bf v} \ra_{\F_{q^{rn/2}}}$ meets $U$ in an $\F_q$-subspace of dimension at most one. Since $\F_{q^n}$ is a subfield of $\F_{q^{rn/2}}$ (recall $r$ even) the same holds for the subspace $\la {\bf v} \ra_{\F_{q^n}}$ and hence $U$ also defines a scattered $\F_q$-linear set in $\PG(W,\F_{q^n})=\PG(r-1,q^n)$.
\end{proof}

Note that the converse of the above result does not hold.

\section{Maximum scattered subspaces and MRD-codes}
The set of $m \times n$ matrices $\F_q^{m\times n}$ over $\F_q$ is a rank metric $\F_q$-space
with rank metric distance defined by $d(A,B) = rk\,(A-B)$ for $A,B \in \F_q^{m\times n}$.
A subset $\cC \subseteq \F_q^{m\times n}$ is called a rank distance code (RD-code for short). The minimum distance of $\cC$ is
\[d(C) = \min_{{A,B \in \cC},\ {A\ne B}} \{ d(A,B) \}.\]

When $\cC$ is an $\F_q$-linear subspace of $\F_q^{m\times n}$, we say that $\cC$ is an $\F_q$-linear code and the
dimension $\dim_q (\cC)$ is defined to be the dimension of $\cC$ as a subspace over $\F_q$. If $d$ is the minimum distance of $\cC$ we say that $\cC$ has parameters $(m,n,q;d)$.

The Singleton bound for an $m\times n$ rank metric code $\cC$ with minimum rank distance $d$ is
\[\#\cC \leq q^{\max \{m,n\}(\min \{m,n\}-d+1)}.\]
If this bound is achieved, then $\cC$ is an MRD-code.
MRD-codes have various applications in communications and cryptography; for instance, see \cite{gabidulin_public-key_1995,koetter_coding_2008}. More properties of MRD-codes can be found in \cite{Delsarte,Gabidulin,gadouleau_properties_2006,morrison_equivalence_2013}.

Delsarte \cite{Delsarte} and Gabidulin \cite{Gabidulin} constructed, independently,  linear MRD-codes over $\F_q$ for any values of $m$ and $n$ and for arbitrary value of the minimum distance $d$. In the literature these are called {\em Gabidulin codes}, even if the first construction is due to Delsarte. These codes were later generalized by Kshevetskiy and Gabidulin in \cite{kshevetskiy_new_2005}, they are the so called {\em generalized Gabidulin codes}.

A generalized Gabidulin code is defined as follows: under a given basis of $\F_{q^n}$ over $\F_q$,  each element $a$ of $\F_{q^n}$ can be written as a (column) vector $\bv(a)$ in $\F_{q}^n$. Let $\alpha_1,\dots,\alpha_m$ be a set of linearly independent elements of $\F_{q^n}$ over $\F_q$, where $m\le n$. Then
\begin{equation}\label{eq:mn_MRD}
\left\{ \left(\bv(f(\alpha_1)), \dots, \bv(f(\alpha_m))\right)^T: f\in \cG_{k,s}
  \right\}
\end{equation}
is the original generalized Gabidulin code, where
\begin{equation}\label{eq:GG}
	\cG_{k,s} = \{f(x)=a_0 x + a_1 x^{q^{s}} + \dots a_{k-1} x^{q^{s(k-1)}}: a_0,a_1,\dots, a_{k-1}\in \F_{q^n} \},
\end{equation} with $n,k,s\in {\mathbb Z}^+$ satisfying $k<n$ and $\gcd(n,s)=1$.

All members of $\cG_{k,s}$ are of the form $f(x)= \sum_{i=0}^{n-1}a_i x^{q^i}$, where $a_i\in \F_{q^n}$. A polynomial of this form is called a \emph{linearized polynomial} (also a $q$-polynomial because its exponents are all powers of $q$). They are equivalent to $\F_q$-linear transformations from $\F_{q^n}$ to itself, i.e., elements of $\mathbb E=\mathrm{End}_{\F_q}(\F_{q^n})$. We refer to \cite[Section 4]{lidl_finite_1997} for their basic properties.

In the literature, there are different definitions of equivalence for rank metric codes; see \cite{berger, morrison_equivalence_2013}.
If $\cC$ and $\cC'$ are two sets of $\GL(U,\F_q)$, where $U$ is an $\F_q$-space of dimension $n$, then up to an isomorphism we may consider $U$ as the finite field $\F_{q^n}$ and it is natural to define equivalence in the language of $q$-polynomials, see \cite{Sh}.
For $\F_q$-linear maps between vector spaces of distinct dimensions we will use the following definition of equivalence.

\begin{definition}
Let $U(n,q)$ and $V(m,q)$ be two $\F_q$-spaces, $n\neq m$, and let $\cC$ and $\cC'$ be two sets of $\F_q$-linear maps from $U$ to $V$. They are {\it equivalent} if there exist two invertible $\F_q$-linear maps $L_1 \in \GL(V,\F_q)$, $L_2\in \GL(U,\F_q)$ and $\rho\in Aut(\F_q)$ such that
$\cC'=\{L_1\circ f^\rho\circ L_2:\ f\in{\cal C}\}$, where $f^\rho(x)=f(x^{\rho^{-1}})^{\rho}$.
\end{definition}

\medskip

Very recently, Sheekey made a breakthrough in the construction of new linear MRD-codes using linearized polynomials \cite{Sh} (see also \cite{LTZ}).


In \cite[Section 4]{Sh}, the author showed that maximum scattered linear sets of $\PG(1,q^n)$ correspond to $\F_q$-linear MRD-codes of dimension $2n$ and minimum distance $n-1$. The number of non-equivalent MRD-codes obtained from a maximum scattered linear set of $\PG(1,q^n)$ was studied in \cite[Section 5.4]{CSMP2016}.

\medskip

Here we extend this result showing that MRD-codes of dimension $rn$ and minimum distance $n-1$ can be constructed from every maximum scattered $\F_q$--linear set of $\PG(r-1,q^n)$, $rn$ even, and we exhibit some relations with Sheekey's construction when $r$ is even.

To this aim, recall that an $\F_q$-subspace $U$ of $\F_{q^{rn}}$ is scattered with respect to $\F_{q^n}$ if it defines a scattered $\F_q$-linear set in $\PG(\F_{q^{rn}},\F_{q^{n}})=\PG(r-1,q^n)$, i.e. $dim_{\F_q}(U\cap \la x\ra_{\F_{q^n}})\leq 1$ for each $x\in\F_{q^{rn}}^*$.

\begin{theorem}
\label{construction}
Let $U$ be an $rn/2$-dimensional $\F_q$-subspace of the $r$-dimensional $\F_{q^n}$-space $V=V(r,q^n)$, $rn$ even, and let
$i=\max\{ \dim_{\F_q} (U \cap \la {\bf v} \ra_{\F_{q^n}}) \colon {\bf v} \in V \}$.
For any $\F_q$-linear function $G \colon V \rightarrow W$, with $W=V(rn/2,q)$ such that $\ker G = U$, if $i<n$, then the pair $(U, G)$ determines
an RD-code $\cC_{U,G}$ (cf. \eqref{CUF}) of dimension $rn$ and with parameters $(rn/2,n,q;n-i)$. Also, $\cC_{U,G}$ is an MRD-code if and only if $U$ is a maximum scattered $\F_q$-subspace with respect to $\F_{q^{n}}$.
\end{theorem}
\begin{proof}
For ${\bf v}\in V$  the set
\[R_{\bf v}:=\{ \lambda \in \F_{q^n} \colon \lambda {\bf v} \in U\}\]
is an $\F_q$-subspace with dimension the weight of the point $\la {\bf v} \ra_{\F_{q^n}}$ in the $\F_q$-linear set $L_U$ of $\PG(V,\F_{q^n})$.
Since $i$ is the maximum weight of the points in $L_U$, it follows that $\dim_{\F_q} R_{\bf v} \leq i$ for each ${\bf v}$.
Also, let $\tau_{\bf v}$ denote the map
\[\lambda \in \F_{q^n} \mapsto \lambda {\bf v} \in V.\]
Direct computation shows that the kernel of $G \circ \tau_{\bf v}$ is $R_{\bf v}$ for each ${\bf v} \in V$ and hence it has
rank at least $n-i$.
It remains to show that $G \circ \tau_{\bf v} \neq G \circ \tau_{\bf w}$ for ${\bf v} \neq {\bf w}$.
Suppose, contrary to our claim, that there exist ${\bf v},{\bf w} \in V$ with ${\bf v} \neq {\bf w}$ and with $G(\lambda {\bf v}) = G(\lambda {\bf w})$ for each $\lambda\in \F_{q^n}$.
Note that ${\bf v} \mapsto G \circ \tau_{\bf v}$ is an $\F_q$-linear map and hence $G(\lambda({\bf v}-{\bf w}))=0$ for each $\lambda\in \F_{q^n}$.
This means $\dim_{\F_q} (\ker G \circ \tau_{{\bf v}-{\bf w}})=n=i$, a contradiction.
Hence
\begin{equation}
\label{CUF}
\cC_{U,G}=\{G \circ \tau_{\bf v} \colon {\bf v}\in V\}
\end{equation}
is an $\F_q$-linear RD-code with dimension $rn$ and with parameters $(rn/2,n,q;n-i)$.
The second part is obvious since $L_U$ is scattered if and only if $i=1$.
\end{proof}

Now we will show that different choices of the function $G$ give rise to equivalent RD-codes. Let's start by proving the following result.

\begin{lemma}
\label{forma}
Let $U$ be an $rn/2$-dimensional $\F_q$-subspace of the $r$-dimensional $\F_{q^{n}}$-space $\F_{q^{rn}}$.
Then there exists $\omega \in \F_{q^{rn}}\setminus \F_{q^{rn/2}}$ such that $$U=\{x+\omega f(x) \colon x\in \F_{q^{rn/2}}\}$$ where $f(x)$ is a $q$-polynomial over $\F_{q^{rn/2}}$.
\end{lemma}
\begin{proof}
Observe that $\F_{q^{rn}}^*=\bigcup_{a\in\F_{q^{rn}}^*}a\F_{q^{rn/2}}^*$ and for any $a,b\in\F_{q^{rn}}^*$ either $a\F_{q^{rn/2}}^*\cap b\F_{q^{rn/2}}^*=\emptyset$ or $a\F_{q^{rn/2}}^*= b\F_{q^{rn/2}}^*$ and the latter case happens if and only if $\frac ab\in\F_{q^{rn/2}}^*$. Since $U^*\cap a\F_{q^{rn/2}}^*$ is either empty or contains at least $q-1$ elements and since $|U^*|=q^{\frac{rn}{2}}-1$, there exist $a,b\in\F_{q^{rn}}^*$, with $\frac ab\notin\F_{q^{rn/2}}$ such that $U^*\cap a\F_{q^{rn/2}}^*=U^*\cap b\F_{q^{rn/2}}^*=\emptyset$. We may assume $a\notin \F_{q^{rn/2}}^*$ and put $\omega:=a$. Then $U\cap \omega \F_{q^{rn/2}}=\{0\}$ and taking into account that $U$ has rank $\frac{rn}2$ and $\{1,\omega\}$ is an $\F_{q^{rn/2}}$--basis of $\F_{q^{rn}}$, we have $U=\{x+\omega f(x) \colon x\in \F_{q^{rn/2}}\}$ for some $q$-polynomial $f$ over $\F_{q^{rn/2}}$.
\end{proof}

Hence, we are able to prove the following

\begin{proposition}\label{prop:equivalenceG}
Let $U$ be an $rn/2$-dimensional $\F_q$-subspace of the $r$-dimen\-sional $\F_{q^n}$-space $V=V(r,q^n)$, $rn$ even, and let $G$ and $\overline{G}$ be two $\F_q$-linear functions determining two RD-codes $\cC_{U,G}$ and $\cC_{U,\overline{G}}$ as in Theorem \ref{construction}. Then $\cC_{U,G}$ and $\cC_{U,\overline{G}}$ are equivalent.
\end{proposition}
\begin{proof}
Up to an isomorphism, we can always assume $V=\F_{q^{rn/2}}\times\F_{q^{rn/2}}$ and $W=\F_{q^{rn/2}}$. Then by Lemma \ref{forma} we have $U=\{(x,f(x)): x \in \mathbb{F}_{q^{\frac{rn}{2}}}\}$, where $f(x)$ is a $q$-polynomial over $\F_{q^{rn/2}}$. Then $G,\overline G:\F_{q^{rn/2}}\times\F_{q^{rn/2}}\rightarrow \F_{q^{rn/2}}$ are two $\F_q$-linear maps such that $U=\ker  G=\ker  \overline G$. We want to show that there exist two permutation $q$-polynomials $H$ and $L$ over $\mathbb{F}_{q^{rn/2}}$ and $\mathbb{F}_{q^n}$, respectively, and $\sigma \in Aut(\mathbb{F}_q)$ such that
\[ \mathcal{C}_{U,\overline{G}}=\{ H \circ (G \circ \tau_{\bf v})^\sigma \circ L \hspace{0.1cm} : \hspace{0.1cm} v \in \mathbb{F}_{q^{rn}} \}.\]
Let $G_0,G_1,\overline{G}_0,\overline{G}_1:\ \mathbb{F}_{q^{rn/2}}\rightarrow \mathbb{F}_{q^{rn/2}}$ be $\mathbb{F}_q-$linear maps such that
$$ G(x,y)=G_0(x)-G_1(y)\mbox{\quad and \quad}\overline{G}(x,y)=\overline{G}_0(x)-\overline{G}_1(y), $$
for all $x,y \in \mathbb{F}_{q^{rn/2}}$.
Since $\ker  G=\ker  \overline G=U$ it can be easily seen that $G_0=G_1\circ f$, $\overline{G}_0=\overline G_1\circ f$ and that $G_1$ and $\overline{G}_1$ are invertible maps. Hence, putting $H=\overline{G}_1\circ G_1^{-1}$, $\sigma=id_{\F_q}$ and $L=id_{\F_{q^n}}$, we have
\[H\circ G\circ \tau_\bv=\overline{G}\circ\tau_\bv,\]
for each $\bv=(x,y)\in V$, and hence the assertion follows.
\end{proof}

First we show some results in the case $r$ even. Starting with the following example for $r=2$, we examine further the codes defined in Theorem \ref{construction}. Later, in Theorem \ref{thm:new} we will also give a different construction of MRD-codes.

\begin{example}\label{example1}
Let $U_f=\{(x,f(x)) \colon x\in \F_{q^n}\}$ be a maximum scattered $\F_q$-subspace of the two-dimensional $\F_{q^n}$-space $V=\F_{q^n} \times \F_{q^n}$, where $f$ is a $q$-polynomial over $\F_{q^n}$. Let
\[G \colon (a,b) \in V \mapsto  f(a)-b \in \F_{q^n}.\]
Then $\ker G = U_f$ and Theorem \ref{construction} with $r=2$ yields the MRD-code consisting of the maps $G \circ \tau_{(a,b)}$, i.e.
\begin{equation}
\label{Fv}
\cC_{U_f,G}=\{ x\in \F_{q^n} \mapsto f(ax)-bx \in \F_{q^n} \colon  (a,b)\in \F_{q^n}\times \F_{q^n} \}.
\end{equation}
Note that the MRD-codes (\ref{Fv}) are the adjoints of the codes constructed by Sheekey in \cite[Sec. 5]{Sh}, see also after Remark \ref{rest}.
\end{example}

\begin{remark}
\label{rest}
Let $U$ be a maximum scattered $\F_q$-subspace of $V=V(2,q^{rn/2})$, $r$ even.
According to Proposition \ref{prop2}, $U$ is also a maximum scattered $\F_q$-subspace of $V$, considered as an $r$-dimensional $\F_{q^n}$-space.
Let $G$ be an $\F_q$-linear $V \rightarrow W=V(rn/2,q)$ map with $\ker G = U$.
When $V$ is viewed as an $\F_{q^n}$-space, then the construction method of Theorem \ref{construction} yields the MRD-code
\begin{equation}
\label{form:CU}
\cC_{U,G}=\{  x\in \F_{q^n} \mapsto G\circ \tau_{\bf v}(x) \in W \colon {\bf v} \in V \}.
\end{equation}
When $V$ is viewed as an $\F_{q^{rn/2}}$-space, then we obtain the MRD-code
\begin{equation}
\label{form:DU}
\cD_{U,G}=\{  x\in \F_{q^{rn/2}} \mapsto G\circ \tau_{\bf v}(x) \in W \colon {\bf v} \in V \}.
\end{equation}
Since $\F_{q^n}$ is a subfield of $\F_{q^{rn/2}}$, the latter code is the restriction of the former one on $\F_{q^n}$.

Conversely, it may happen, even if $r$ is even, that an $\F_q$-subspace $U$ of $V=V(r,q^n)$ of rank $rn/2$ is scattered with respect to $\F_{q^n}$ whereas it is not scattered when $V$ is considered  as a $2$-dimensional $\F_{q^{rn/2}}$-space. Arguing as above, the MRD-code $\cC_{U,G}$ described in (\ref{form:CU}) is the restriction of the RD-code $\cD_{U,G}$ described in (\ref{form:DU}).
\end{remark}

Let $\omega_\alpha$ be the map $\F_{q^{rn/2}}\rightarrow \F_{q^{rn/2}}$ defined by the rule $x\mapsto \alpha x$.
By $(\omega_\alpha + \omega_{\beta} \circ f)\mid_{\F_{q^n}}$ we denote the restriction of the corresponding function over $\F_{q^n}$.
From Example \ref{example1} and from Remark \ref{rest} it follows that if $r$ is even and $U_f=\{(x,f(x)) \colon x\in \F_{q^{rn/2}}\}$ is a maximum scattered $\F_q$-subspace of $\F_{q^{rn/2}}^2$ considered as an $r$-dimensional $\F_{q^n}$-space, then the MRD-code (cf. (\ref{Fv}), (\ref{form:CU}) and (\ref{form:DU}))
\[\cC_f=\{(\omega_\alpha + f \circ \omega_{\beta})\mid_{\F_{q^n}} \colon \alpha, \beta \in \F_{q^{rn/2}}\}\]
is the restriction on $\F_{q^n}$ of the MRD-code
\[\cD_f=\{(\omega_\alpha + f \circ \omega_{\beta}) \colon \alpha, \beta \in \F_{q^{rn/2}}\}.\]
The next result shows that $\{(\omega_\alpha + \omega_{\beta} \circ f)\mid_{\F_{q^n}}\colon \alpha, \beta \in \F_{q^{rn/2}}\}$ is also an MRD-code with the same parameters as $\cC_f$. For $r=2$ this is exactly the code defined by Sheekey.

\begin{theorem}\label{thm:new}
Let $r$ be even and $U_f:=\{(x,f(x)) \colon x\in \F_{q^{rn/2}}\}$ be a maximum scattered $\F_q$-subspace of $\F_{q^{rn/2}}^2$ considered as $V(r,q^n)$,
where $f$ is a $q$-polynomial over $\F_{q^{rn/2}}$.
Then ${\cal S}_f:=\{(\omega_\alpha + \omega_{\beta} \circ f)\mid_{\F_{q^n}}\colon \alpha, \beta \in \F_{q^{rn/2}}\}$ is an
MRD-code with parameters $(rn/2,n,q;n-1)$.
\end{theorem}
\begin{proof}
Since $U_f$ is scattered, the following holds.
If $(x,f(x))=\lambda(y,f(y))$ with $\lambda \in \F_{q^n}$, then $\lambda\in \F_q$, so for each $y\in \F_{q^{rn/2}}^*$
\begin{equation}
\label{scatt}
f(\lambda y)=\lambda f(y) \text{ with } \lambda \in \F_{q^n} \text{ implies } \lambda \in \F_q.
\end{equation}
It also follows that for each $y\in \F_{q^{rn/2}}^*$ we have
\begin{equation}
\label{scatt2}
f(\lambda y)/\lambda y = f(y)/y \text{ for some } \lambda \in \F_{q^n}^* \text{ if and only if } \lambda \in \F_q^*.
\end{equation}

First we show that $(\alpha x + \beta f(x))\mid_{\F_{q^n}}=0$ implies $\alpha=\beta=0$.
Suppose the contrary. If $\beta\neq 0$, then $f(x)= x t$, with $t=-\alpha/\beta$ for each $x\in \F_{q^n}$, contradicting \eqref{scatt}.
If $\beta=0$, then clearly also $\alpha=0$. It follows that $|{\cal S}_f|=q^{rn}$.

The $\F_q$-linear map $(\alpha x + \beta f(x))\mid_{\F_{q^n}}$ has rank less than $n-1$ if and only if $\beta\neq 0$ and there exist $x,y\in \F_{q^n}^*$ such that $\la x \ra_{\F_q} \neq \la y \ra_{\F_q}$ and $f(x)/x=f(y)/y=-\alpha/\beta$.
But then for $\lambda:=x/y \in \F_{q^n}\setminus \F_q$ we have $f(\lambda y)/\lambda y=f(y)/y$ contradicting \eqref{scatt2}.
\end{proof}

Sheekey in \cite[Theorem 8]{Sh} showed that when $r=2$ the two $\F_q$--vector subspaces $U_f$ and $U_g$ defined as in Theorem \ref{thm:new} are equivalent under the action of the group $\Gamma \mathrm{L}(2,q^n)$ if and only if $\cS_f$ and $\cS_g$ are equivalent as MRD-codes. Here we will show that the same result is not true when we consider the restriction codes.
To show this we will need the following two examples, where non-equivalent $\F_q$-subspaces yield the same MRD-code.

\begin{example}
\label{const1}
Consider $U_f=\{(x,f(x)) \colon x \in \F_{q^{tn}}\}$, with $t \geq 1$, $n \geq 3$ and
with $f \colon \F_{q^{tn}}\rightarrow \F_{q^{tn}}$ an invertible $\F_{q^n}$-semilinear map with associated automorphism $\sigma \in \mathrm{Aut}(\F_{q^n})$ such that $\mathrm{Fix}(\sigma)=\F_q$. Then $L_{U_f}$ is a scattered $\F_q-$linear set of pseudoregulus type in $\PG(2t-1,q^n)$ (cf. \cite[Sec. 3]{LuMaPoTr2014}).
With this choice of $f$, we get
\[\cS_f=\{(\omega_\alpha+\omega_{\beta} \circ id^\sigma)\mid_{\F_{q^n}} \colon \alpha, \beta \in \F_{q^{tn}}\}.\]
Indeed, for every $\lambda \in \F_{q^n}$ we have
$(\omega_\alpha+\omega_\beta \circ f)(\lambda)=\alpha \lambda+\beta f(\lambda)=\alpha\lambda+\beta\lambda^\sigma f(1)$.
\end{example}

\begin{example}
\label{const2}
Let $W=\{(x,y,x^q,y^{q^h}) \colon x,y \in \F_{q^n}\}$, with $n \geq 5$, $1 < h < n-1$ and with $\gcd(h,n)=1$.
Then $W$ is a scattered $\F_q$-subspace of $V(4,q^n)$ and it defines an $\F_q$-linear set $L_W$ of $\PG(3,q^n)$, which is not of pseudoregulus type, see \cite[Proposition 2.5]{LaMaPoTr2013}.
We may consider $V(4,q^n)$ as $\F_{q^{2n}}\times \F_{q^{2n}}$. Take $\omega \in \F_{q^{2n}}\setminus \F_{q^n}$, so $\{1,\omega\}$ is an $\F_{q^n}$-basis of $\F_{q^{2n}}$ and
\[W=\{(x+\omega y, x^q+\omega y^{q^h}) \colon x,y \in \F_{q^n}\}.\] Direct computations show that $W=\{(z,g(z)) \colon z \in \F_{q^{2n}} \}=U_g$, where $g$ is the $q$-polynojmial over $\F_{q^{2n}}$ of the form
$$g(z)=a_1z^q+a_h z^{q^h}+(1-a_1)z^{q^{n+1}}-a_hz^{q^{n+h}},$$
with $a_1= \frac{\omega^{q^{n+1}}}{\omega^{q^{n+1}}-\omega^q} $ and $a_h= \frac{1}{\omega^{q^h-1}-\omega^{q^{h+n}-1}}$.
Hence $g(z)\mid_{\F_{q^n}}=z^q$,
so
\[\cS_g=\{(\omega_\alpha + \omega_\beta \circ id^q)\mid_{\F_{q^n}} \colon \alpha, \beta \in \F_{q^{2n}}\}.\]
\end{example}

\begin{theorem}
In $V(4,q^n)$, $n\geq 5$, there exist two non-equivalent maximum scattered $\F_q$-subspaces $U_f$ and $U_g$ such that the codes $\cS_f$ and $\cS_g$ coincide.
\end{theorem}
\begin{proof}
In Example \ref{const1} take $t=2$ and $\sigma \colon x \mapsto x^q$. Then we obtain the same code as in Example \ref{const2}, while
the corresponding subspaces are non-equivalent because of \cite[Proposition 2.5]{LaMaPoTr2013}.
\end{proof}

Let now $r$ be odd and $n=2t$. Some of the known families of maximum scattered $\F_q$-subspaces are given in the $r$-dimensional $\F_{q^{2t}}$-space $V=\F_{q^{2rt}}$ and they are of the form
\begin{equation}
U_f:=\{x\omega + f(x) \colon x\in \F_{q^{rt}}\},
\end{equation}
with $\omega \in \F_{q^{2t}} \setminus \F_{q^t}$ and with $\omega^2=\omega A_0+A_1$, $A_0,A_1\in\F_{q^t}$.
In this case we show an explicit construction of $\F_q$-linear MRD-codes with parameters $(rt,2t,q;2t-1)$ obtained from
Theorem \ref{construction}.
Indeed, in this case $\{\omega, 1\}$ is an $\F_{q^t}$-basis of $\F_{q^{2t}}$ and also an $\F_{q^{rt}}$-basis of $\F_{q^{2rt}}$.
Then we can write any element $\lambda\in\F_{q^{2t}}$ as $\lambda= \lambda_0 \omega +\lambda_1$, with $\lambda_0, \lambda_1 \in\F_{q^t}$.
We fix $G \colon \F_{q^{2rt}} \rightarrow \F_{q^{rt}}$ as the map $x \omega + y \mapsto f(x)-y$.
For each $v=v_0 \omega + v_1 \in \F_{q^{2rt}}$ the map $\tau_v \colon \F_{q^{2t}} \rightarrow \F_{q^{2rt}}$ is as follows
\[\lambda \mapsto \lambda_0 v_0 A_1 + \lambda_1 v_1 + \omega (\lambda_0 v_1 + \lambda_1 v_0 + \lambda_0 v_0 A_0),\]
and $\tau_v$ can be viewed as a function defined on $\F_{q^t}\times\F_{q^t}$.
Then the associated MRD-code consists of the following maps:
\[G \circ \tau_v \colon (x,y)\in \F_{q^t}\times \F_{q^t} \mapsto f(\lambda_0 v_1 + \lambda_1 v_0 + \lambda_0 v_0 A_0)- \lambda_0 v_0 A_1 - \lambda_1 v_1.\]

\begin{example}\label{ex1}
Put $f(x):=ax^{q^i}$,  $a\in\F_{q^{rt}}^*$, $1\leq i\leq rt-1$, $r$ odd. For any $q\geq 2$ and any integer $t\geq 2$ with $\gcd(t,r)=1$, such that\
\begin{enumerate}[(i)]
\item $\gcd(i,2t)=1$ and $\gcd(i,rt)=r$,
\item $\N_{q^{rt}/q^r}(a)\notin\F_q$,
\end{enumerate}
from \cite[Theorem 2.2]{BGMP2015}, we get the $\F_q$-linear MRD-code with dimension $2rt$ and parameters $(2t,rt,q;2t-1)$:
\[ \{ F_v \colon v=\omega v_0+v_1,\, v_0, v_1 \in \F_{q^{rt}} \}, \]
where $F_v \colon \F_{q^t}\times\F_{q^t}\rightarrow\F_{q^{rt}}$ is defined by the rule
\begin{equation}\label{for:cod1}
F_v(x,y)=x^{q^i}a(A_0^{q^i}v_0^{q^i}+v_1^{q^i})-xA_1v_0+y^{q^i}av_0^{q^i}-yv_1.
\end{equation}
Note that, since $\gcd(i,rt)=r$, the above MRD-code is $\F_{q^r}$-linear as well, since for each $\mu\in\F_{q^r}$ and $v\in\F_{q^{2rt}}$ we have $\mu F_v=F_{\mu v}$.
\end{example}

\begin{example}\label{ex2}
Put $f(x):=ax^{q^i}$,  $a\in\F_{q^{rt}}^*$, $1\leq i\leq rt-1$, $r$ odd. For any prime power $q\equiv 1 \pmod r$ and any integer $t\geq 2$, such that\
\begin{enumerate}[(i)]
\item $\gcd(i,2t)=\gcd(i,rt)=1$,
\item $\left(\N_{q^{rt}/q}(a)\right)^{\frac{q-1}{r}}\ne 1$,
\end{enumerate}
from \cite[Theorem 2.3]{BGMP2015}, we get the $\F_q$-linear MRD-code with dimension $2rt$ and parameters $(2t,rt,q;2t-1)$:
\[ \{ F_v \colon v=\omega v_0+v_1,\, v_0, v_1 \in \F_{q^{rt}} \}, \]
where $F_v \colon \F_{q^t}\times\F_{q^t}\rightarrow\F_{q^{rt}}$ is defined by the same rule as \eqref{for:cod1}.
\end{example}

\begin{example}\label{ex3}
Put $f(x):=ax^{q^i}+bx^{q^{2t+i}}$,  $a,b\in\F_{q^{3t}}^*$, $1\leq i\leq 3t-1$ (here $r=3$). For any $q\geq 2$ and any integer $t\geq 2$ with $\gcd(i,2t)=1$ choosing $a,b$ as in the proof of Theorem \ref{thm:1}, we get the $\F_q$-linear MRD-code with dimension $6t$ and parameters $(2t,3t,q;2t-1)$:
\[ \{ F_v \colon v=v_0+\omega v_1,\, v_0, v_1 \in \F_{q^{3t}} \}, \]
where $F_v:\, \F_{q^t}\times\F_{q^t}\rightarrow\F_{q^{3t}}$ is defined by the rule
\[F_v(x,y)=x^{q^i}(aA_0^{q^i}v_0^{q^i}+av_1^{q^i}+bA_0^{q^i}v_0^{q^{2t+i}}+bv_1^{q^{2t+i}})+\]
\[y^{q^i}(av_0^{q^i}+bv_0^{q^{2t+i}})-xA_1v_0-yv_1.\]
\end{example}

Applying \cite[Theorem 3.1]{BGMP2015} one can construct other MRD-codes after decomposing $V(r,q^n)$
into a direct sum of $\F_{q^n}$-subspaces of dimensions 2 and 3 and choosing for each of them a maximum scattered subspace.

\bigskip

\noindent Bence Csajb\'ok\\
Dipartimento di Matematica e Fisica,\\
Universit\`a degli Studi della Campania ``Luigi Vanvitelli'',\\
I--\,81100 Caserta, Italy\\
and \\
MTA--ELTE Geometric and Algebraic Combinatorics Research Group, \\
E\"otv\"os Lor\'and University, \\
H--\,1117 Budapest, P\'azm\'any P\'eter S\'et\'any 1/C, Hungary\\
{{\em csajbok.bence@gmail.com}}

\bigskip

\noindent Giuseppe Marino, Olga Polverino and Ferdinando Zullo\\
Dipartimento di Matematica e Fisica,\\
Universit\`a degli Studi della Campania ``Luigi Vanvitelli'',\\
I--\,81100 Caserta, Italy\\
{{\em giuseppe.marino@unina2.it}, {\em olga.polverino@unina2.it}, {\em ferdinando.zullo@unina2.it}}

\end{document}